\documentclass[preprint,12pt]{elsarticle}

\usepackage{amssymb}
\usepackage{amsfonts}
\usepackage{amsmath}
\usepackage{amsthm}
\usepackage[linesnumbered,ruled]{algorithm2e}
\usepackage{url}
\newcommand\RR{{\mathbb R}}

\DeclareMathOperator{\TS}{TS}
\DeclareMathOperator{\solve}{solve}
\DeclareMathOperator{\sg}{sg}
\DeclareMathOperator{\man}{man}

\newtheorem{thm}{Theorem}
\newtheorem{lem}[thm]{Lemma}
\newtheorem{prop}[thm]{Proposition}
\newtheorem{cor}[thm]{Corollary}

\newdefinition{defn}[thm]{Definition}
\newdefinition{exmp}[thm]{Example}
\newdefinition{conj}[thm]{Conjecture}
\newdefinition{claim}{Claim}
\newdefinition{construction}[thm]{Construction}
\newdefinition{rem}[thm]{Remark}
\newdefinition{qst}[thm]{Question}

\journal{Discrete Optimization}

\begin{document}

\begin{frontmatter}

%% Title, authors and addresses

%% use the tnoteref command within \title for footnotes;
%% use the tnotetext command for the associated footnote;
%% use the fnref command within \author or \address for footnotes;
%% use the fntext command for the associated footnote;
%% use the corref command within \author for corresponding author footnotes;
%% use the cortext command for the associated footnote;
%% use the ead command for the email address,
%% and the form \ead[url] for the home page:
%%
%% \title{Title\tnoteref{label1}}
%% \tnotetext[label1]{}
%% \author{Name\corref{cor1}\fnref{label2}}
%% \ead{email address}
%% \ead[url]{home page}
%% \fntext[label2]{}
%% \cortext[cor1]{}
%% \address{Address\fnref{label3}}
%% \fntext[label3]{}

\title{Searching for Realizations of Finite Metric Spaces in Tight Spans}

\author[shvm]{Sven Herrmann}
\ead{sherrmann@mathematik.tu-darmstadt.de}

\author[shvm]{Vincent Moulton\corref{cor}}
\ead{vincent.moulton@cmp.uea.ac.uk}

\author[as]{Andreas Spillner}
\ead{andreas.spillner@uni-greifswald.de}

\cortext[cor]{Corresponding author}

\address[shvm]{School of Computing Sciences, University of East Anglia, Norwich, NR4 7TJ, United Kingdom}
\address[as]{Department of Mathematics and Computer Science, University of Greifswald, 17487 Greifswald, Germany}

%% use optional labels to link authors explicitly to addresses:
%% \author[label1,label2]{<author name>}
%% \address[label1]{<address>}
%% \address[label2]{<address>}

\begin{abstract}
An important problem that commonly arises in areas 
such as internet traffic-flow analysis, phylogenetics 
and electrical circuit design, is to find
a representation of any given metric $D$ on a finite set 
by an edge-weighted graph, such that the total edge 
length of the graph is minimum over all such 
graphs. Such a graph is called an {\em optimal realization}
and finding such realizations 
is known to be NP-hard.
Recently Varone presented a heuristic greedy algorithm 
for computing optimal realizations.
Here we present an alternative heuristic
that exploits the relationship between realizations
of the metric $D$ and its so-called tight span $T_D$.
The tight span $T_D$
is a canonical polytopal complex 
that can be associated to $D$, and our approach 
explores parts of \(T_D\) for realizations 
in a way that is similar to the classical simplex algorithm.
We also provide computational results illustrating 
the performance of our approach for different 
types of metrics, including \(l_1\)-distances and two-decomposable 
metrics for which it is provably
possible to find optimal realizations in 
their tight spans.
\end{abstract}

\begin{keyword}
%% keywords here, in the form: keyword \sep keyword
combinatorial optimization \sep metric \sep graph \sep realization \sep tight span
%% MSC codes here, in the form: \MSC code \sep code
%% or \MSC[2008] code \sep code (2000 is the default)

\end{keyword}

\end{frontmatter}

%\linenumbers

%% main text
\section{Introduction}

An important problem that 
commonly arises in areas such as 
internet traffic-flow analysis \cite{chu-gar-gra-01a}, 
phylogenetics \cite{ban-dre-92b}
and electrical circuit design \cite{hak-yau-64a}, 
is to realize any given metric $D$ on some finite set $X$ 
by an edge-weighted graph with $X$ 
labeling its vertex set, often with the additional
requirement that the total edge length of the graph is minimum.
This can be useful, for example, for visualizing the metric, or
for trying to better understand its structural properties. 
More formally this optimization problem 
can be stated as follows. 
A \emph{realization} \((G,\omega,\tau)\)
of $D$ is a connected graph \(G=(V,E)\)
with vertex set $V$ and edge set $E$, together with 
an edge-weighting 
\(\omega: E \rightarrow \mathbb{R}_{>0}\) 
and a labeling map \(\tau: X \rightarrow V\) such that,
for all \(x,y \in X\), \(D(x,y)\) equals \(D_G(\tau(x),\tau(y))\),
that is, the length of a shortest path from \(\tau(x)\)
to \(\tau(y)\) in \(G\) (cf. Figure~\ref{figure:examples:or}(a) and (b)).
The problem then is to find an 
\emph{optimal} realization of $D$, that is,
a realization of $D$ that has minimum total edge length
over all possible realizations of $D$.

\begin{figure}[h]
\centering
\includegraphics[scale=1.0]{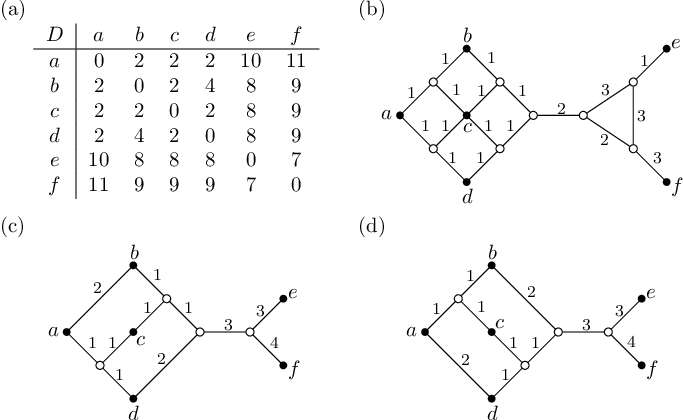}
\caption{(a) A metric \(D\) on \(X=\{a,b,c,d,e,f\}\).
         (b) A realization of \((X,D)\) that is not optimal.
             Vertices associated with an element of \(X\) are
             drawn as black dots, the remaining vertices are
             drawn as empty circles.
         (c),(d) Two optimal realizations of \((X,D)\).}
\label{figure:examples:or}
\end{figure}

Early work on optimal realizations started with 
\cite{hak-yau-64a} (see also \cite{var-06a} for a 
comprehensive list of references), which focused mainly
on special classes of metrics such as, for example,
those that admit an optimal realization 
where the underlying graph is a tree (so-called 
\emph{treelike} metrics).
Subsequently it was found that 
{\em every} metric \(D\) on a finite set \(X\) 
has an optimal realization
\cite{imr-sim-84a}, although 
this need not be unique 
(cf. Figure~\ref{figure:examples:or}(c) and (d)).
There even always exists an optimal realization
of \((X,D)\) with \(O(|X|^4)\) vertices \cite[p. 392]{dre-84},
which implies that there is an exhaustive
algorithm to search for an optimal realization.
However, it was also shown that computing an optimal realization is NP-hard 
\cite{alt-88a,win-84a}. More recently,
there has been renewed interest in computational aspects
of this problem. For example, 
in \cite{her-var-07a,her-var-08a} (see also 
\cite{dre-hub-al-10}) a way to 
break up the problem of computing an optimal
realization into subproblems using so-called 
\emph{cut points} is presented, 
and in \cite{var-06a} a heuristic is presented for
computing optimal realizations.

Here we present an alternative heuristic
for systematically computing optimal realizations that 
exploits the relationship between optimal
realizations of a metric $D$  
and its so-called \emph{tight span} \(T_D\) \cite{dre-84,isb-64}.
In brief (see Section~\ref{section:preliminaries} for details),
$T_D$ is a polytopal complex (essentially a union 
of polytopes) that can be canonically associated to $D$ 
which is itself a (non-finite) metric space and into 
which the metric $D$ can be canonically embedded. 
Remarkably, in \cite{dre-84} it is shown that the 
1-skeleton \(G_D\) of $T_D$ (i.e., 
the edge-weighted graph formed essentially by taking 
all of the 0- and 1-dimensional faces of \(T_D\))
is always a realization of~$D$.
Moreover, Dress conjectured 
\cite[(3.20)]{dre-84} that some optimal realization of
$D$ can always be obtained by removing
some set of edges from~\(G_D\).  

While Dress' conjecture is still open for metrics in general,
recently it has been shown to hold for the class
of so-called \emph{two-decomposable} metrics
\cite[Theorem~1.2]{her-koo-11a}, a class which includes
treelike metrics and \(l_1\)-distances between 
points in the plane (see Section~\ref{section:two:decomposable} 
for more details). 
In particular, this and 
the aforementioned result in \cite{dre-84} suggest
that it could be useful to consider \(G_D\) 
as a ``search space'' in which to 
look for some optimal realization of $D$ 
(or at least some 
interesting realization of $D$ which has 
relatively small total edge length).

Guided by this principle, 
given an arbitrary finite metric $D$,
in Section~\ref{section:algorithm:tight:span}
we propose a heuristic 
for computing a realization of $D$ 
that is a subgraph of~\(G_D\). This heuristic
explores parts of \(T_D\) in a way similar to
the classical simplex algorithm \cite{dan-63}.
Moreover, it does not explicitly compute \(G_D\), whose
vertex set can have cardinality that is 
exponential in \(|X|\) (see e.g.
\cite{her-jos-07a} for some explicit bounds).
We also show that the heuristic is guaranteed to
find optimal realizations for some simple types
of metrics.

Since, as mentioned above, the problem of finding optimal
realizations is NP-hard, we assess the performance of
our new heuristic using two strategies. First, 
we consider a special instance of the problem 
where we take metrics to be \(l_1\)-distances between points in the 
plane. In Section~\ref{section:mmn:or:connection} we show 
that finding optimal realizations of such a metric $D$
in $G_D$ is equivalent to the 
so-called \emph{minimum Manhattan network} problem
(which was also recently shown to be NP-hard \cite{chi-guo-09a}).
This allows us to compare the realizations computed by 
our heuristic with realizations computed using a
mixed integer linear program (MIP) for the minimum Manhattan network
problem presented in \cite{BenkertWWS06} (see also \cite{kna-spi-11a}
for a comprehensive list of references on other approaches for solving
this well-studied problem).
Second, in Section~\ref{section:minimal:subrealizations} 
we describe a mixed integer program (MIP) for computing a 
\emph{minimal subrealization} of a realization of some metric,
that is, a subrealization with minimum 
total edge length. This allows us to 
obtain some impression of how close 
the realizations computed by our heuristic are
to a minimal subrealization of \(G_D\) 
in case \(|X|\) is not too large. Moreover, in 
case the metric is two-decomposable, a minimal subrealization
of \(G_D\) is (by the aforementioned result 
in \cite{her-koo-11a}) an optimal realization
and so we can compare the realizations 
computed by our new heuristic with optimal ones
for this special class of metrics.

Based on these considerations, in 
Section~\ref{section:computational:experiments}
we present simulations for \(l_1\)-distances, 
two-decomposable metrics and random metrics
to assess the performance of our heuristic.
An implementation of this heuristic is freely available 
for download at
\url{www.uea.ac.uk/cmp/research/cmpbio/CoMRiT/}.
This includes the algorithm for efficiently computing 
cut points as described in \cite{dre-hub-al-10} and auxiliary
programs that allow to generate the MIP description for the 
minimum Manhattan network problem,
as well as for the problem of computing a minimal subrealization 
so that they can be solved using existing MIP solvers
(we used the solver that is part of the GNU linear programming kit
(\url{www.gnu.org/software/glpk/}) in our experiments). 
We conclude the paper with a brief discussion
of some possible future directions in Section~\ref{section:discussion}.

%%%%%%%%%%%%%%%%%%%%%%%%%%%%%%%%%%%%%%%%%%%%%%%%%%%%%%%%%
\section{Preliminaries}
\label{section:preliminaries}
%%%%%%%%%%%%%%%%%%%%%%%%%%%%%%%%%%%%%%%%%%%%%%%%%%%%%%%%%

In this section, we first recall the formal definition of the 
tight span of a metric, a concept that has been discovered
and re-discovered several times in the literature
(see e.g. \cite{chr-lar-94a,dre-84,isb-64}).
We also recall some facts concerning tight spans
and optimal realizations that will be used
later on (for more on this see e.g. \cite[Chapter 5]{dre-hub-12a}). 

%%%%%%%%%%%%%%%%%%%%%%%%%%%%%%%%%%%%%%%%%%%%%%%%%%%%%%%%%%
\subsection{Some tight span theory}
\label{section:little:bit:about:tight:spans}
%%%%%%%%%%%%%%%%%%%%%%%%%%%%%%%%%%%%%%%%%%%%%%%%%%%%%%%%%%

A \emph{finite metric space} is a pair $(X,D)$
consisting of a finite non-empty set \(X\) and
a symmetric bivariate map \(D: X \times X \rightarrow \mathbb{R}_{\geq 0}\)
such that \(D(x,x) = 0\) and \(D(x,z) \leq D(x,y) + D(y,z)\)
for all \(x,y,z \in X\). To emphasize that
\(D(x,y) = 0\) does not necessarily imply \(x=y\), such a map \(D\)
is often called a \emph{pseudometric}, but we will
simply refer to \(D\) as a  \emph{metric} here.
A map \(h:X \rightarrow X'\) from a metric space
\((X,D)\) into a metric space \((X',D')\) is an
\emph{isometric embedding} if \(D'(h(x),h(y)) = D(x,y)\)
for all \(x,y \in X\).

Now, given 
any finite metric space \((X,D)\), the \emph{tight span} \(T_D\)
is defined to be the polytopal complex (see e.g. \cite{kle-kle-99a}) 
that is the union of the bounded faces of the polyhedron
\[
P_D :=\{f\in\RR^X:f(x)+f(y)\geq D(x,y) \text{ for all } x,y\in X\}.
\]
Viewed as a subset of \(\mathbb{R}^X\), 
\(T_D\) can be endowed with the \(l_{\infty}\)-metric
which is defined by 
\[
D_{\infty}(f,g) = \max\{|f(x) - g(x)| : x \in X\}
\] 
for all \(f,g \in T_D\)
so that \((T_D,D_{\infty})\) is also a (non-finite!)
metric space.
Note that there exists a canonical isometric embedding 
of \((X,D)\) into \((T_D,D_{\infty})\), the 
so-called \emph{Kuratowski embedding} \cite{kur-35a},
that maps every \(x \in X\)
to \(k_x:X \rightarrow \mathbb{R} : y \mapsto D(x,y)\).
Note that the map \(k_x\) is a 0-dimensional face (or
vertex) of \(T_D\) for every \(x \in X\) and, therefore,
it is contained in the 1-skeleton~\(G_D\).

Later we will use the fact that the tight span 
can be viewed as a hull of the given metric space
similar to the convex hull associated to a set of
points in Euclidean space. To make this more
precise, define a map \(h:X \rightarrow X'\) from a metric space
\((X,D)\) into a metric space \((X',D')\) to be
\emph{non-expansive} if \(D'(h(x),h(y)) \leq D(x,y)\)
for all \(x,y \in X\), and a metric space \((X',D')\)
to be \emph{injective} if for every metric space \((X,D)\)
and every subset \(Y \subseteq X\) any non-expansive
map of the subspace \((Y,D|_Y)\) into \((X',D')\)
can be extended to a non-expansive map of \((X,D)\)
into \((X',D')\). 
The tight span satisfies the following universal property
\cite{dre-84,isb-64}:

\begin{lem}
\label{lemma:induced:isometric:embedding}
Any isometric embedding of a metric space \((X,D)\)
into an injective metric space \((X',D')\) can be
extended to an isometric embedding of \((T_D,D_{\infty})\)
into \((X',D')\).
\end{lem}

%%%%%%%%%%%%%%%%%%%%%%%%%%%%%%%%%%%%%%%%%%%%%%%%%%%%%%%%%%%
\subsection{Tight spans and optimal realizations}
\label{section:tight:span:or}
%%%%%%%%%%%%%%%%%%%%%%%%%%%%%%%%%%%%%%%%%%%%%%%%%%%%%%%%%%%

We now present a key relationship between
realizations and tight spans that was
first discovered by Dress. Let $(X,D)$ 
be an arbitrary finite metric space and
\(G_D=(V_D,E_D)\) the graph that forms the
1-skeleton of \(T_D\). Defining the map 
\(\omega_D:E_D \rightarrow \mathbb{R}_{\geq 0}\) 
by putting $\omega_D(\{u,v\})=D_\infty(u,v)$
for all edges \(\{u,v\}\) of \(G_D\) and
the map \(\tau_D:X \rightarrow V_D\) by putting \(\tau_D(x) = k_x\)
for all \(x \in X\),
it is shown in \cite[Theorem~5]{dre-84} that
\((G_D=(V_D,E_D),\omega_D,\tau_D)\) is 
a realization of \((X,D)\) 
(see also \cite[Theorem 5.15]{dre-hub-12a}).
Moreover, in \cite{dre-84} it is shown that, for 
any optimal realization $(G=(V,E),\omega,\tau))$ 
of $(X,D)$, there exists a map $h:V \to T_D$
with \(h(\tau(x)) = k_x\) for all \(x \in X\)
that preserves certain distances, that is, 
\(\omega(\{u,v\}) = D_{\infty}(h(u),h(v))\) for all 
edges~\(\{u,v\} \in E\).
While this suggests that every optimal realization
of $(X,D)$ is somehow 'contained' in $T_D$, in \cite{alt-88a} it was
shown that there exists an infinite
family of optimal realizations of a certain
metric $D^*$ on six points, for
which no member is isomorphic to some subrealization
of the 1-skeleton of $T_{D^*}$.
Still, as mentioned in the introduction,
it is not known whether 
or not there always exists some
optimal realization of \((X,D)\) that is
a subrealization of \((G_D=(V_D,E_D),\omega_D,\tau_D)\).

%%%%%%%%%%%%%%%%%%%%%%%%%%%%%%%%%%%%%%%%%%%%%%%%%%%%%%%%%%%%%%%%
\section{Two-decomposable metrics}
\label{section:two:decomposable}
%%%%%%%%%%%%%%%%%%%%%%%%%%%%%%%%%%%%%%%%%%%%%%%%%%%%%%%%%%%%%%%%

Before we present our heuristic in the next section,
we shall briefly consider a
special class of finite metrics \(D\),
the two-decomposable metrics,
for which it is \emph{known} that \(G_D\) always contains a
subrealization that is an optimal realization of \(D\).
As mentioned in the introduction, these metrics
are of interest as we can in principle compute optimal realizations
for them exactly and thus measure the accuracy of our
heuristic for computing realizations for small metric spaces.

We first need to recall some relevant concepts. A
\emph{split} \(S\) of a finite set $X$ 
is a bipartition \(\{A,B\}\) of \(X\)
into two non-empty subsets \(A\) and \(B\), also  denoted
by \(A|B\). For any \(x \in X\), that set in \(S\) that contains
\(x\) is denoted by \(S(x)\) and the other set by \(\overline{S}(x)\).
Two splits \(A|B\) and \(A'|B'\) of \(X\) are \emph{compatible}
if at least one of the intersections
\(A \cap A'\), \(A \cap B'\), \(B \cap A'\) and \(B \cap B'\) is
empty. Otherwise the two splits are \emph{incompatible}.
A set \(\Sigma\) of splits of \(X\) is called a \emph{split system} (on \(X\)).
A split system \(\Sigma\) is \emph{two}-\emph{compatible} if there is
no subset \(\Sigma' \subseteq \Sigma\) with \(|\Sigma'| = 3\) 
and any two distinct splits in \(\Sigma'\) are incompatible.

Now, for any split \(S\) of \(X\), define the metric \(D_S\) on \(X\)
putting, for all \(x,y \in X\), \(D_S(x,y) = 0\) if \(S(x) = S(y)\) and
\(D(x,y) = 1\) otherwise. A metric \(D\) on \(X\) is
\emph{two}-\emph{decomposable} if there exists a two-compatible
split system \(\Sigma\) on \(X\) and a weighting 
\(\lambda:\Sigma \rightarrow \mathbb{R}_{> 0}\)
with \(D = \sum_{S \in \Sigma} \lambda(S) \cdot D_S\).
We also say that \(D\) is \emph{induced} by \(\Sigma\)
and the weighting \(\lambda\). 
Later we will use the following result 
\cite[Theorem~1.2]{her-koo-11a}:

\begin{thm}
\label{theorem:vertex:embedding:optimal:realization:into:tight:span}
Let \(D\) be a two-decomposable metric on \(X\). 
Then there always exists an optimal realization
that is a subrealization of \((G_D,\omega_D,\tau_D)\).
In particular, there exists an optimal 
realization \((G=(V,E),\omega,\tau)\) of \((X,D)\)
such that there exists an injective map \(h:V \rightarrow T_D\)
with \(w(\{u,v\}) = D_{\infty}(h(u),h(v))\) for all edges
\(\{u,v\} \in E\) and \(h(\tau(x)) = k_x\) for all \(x \in X\).
\end{thm}

We illustrate this theorem in Figure~\ref{figure:example:tight:span}.
More specifically, the metric $D$ in 
Figure~\ref{figure:examples:or}(a) is two-decomposable,
and its tight span is depicted in Figure~\ref{figure:example:tight:span}(a).
The realization $G_D$ is pictured 
in Figure~\ref{figure:example:tight:span}(b), and a
two-compatible split system 
associated to $D$ is given in Figure~\ref{figure:example:tight:span}(c).
Note that both of the optimal realizations for $D$
given in Figure~\ref{figure:examples:or}(c) and (d)
can be obtained from $G_D$ by removing precisely two edges.

\begin{figure}
\centering
\includegraphics[scale=1.0]{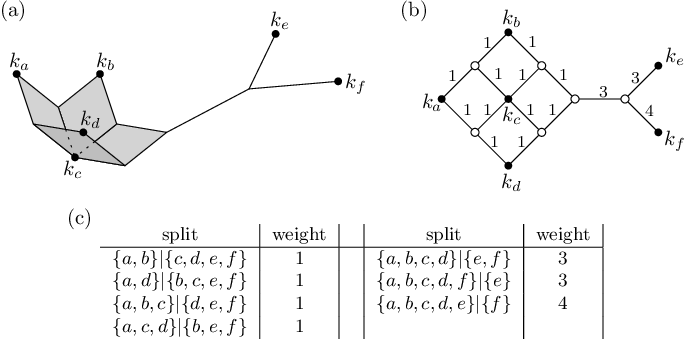}
\caption{(a) The tight span \(T_D\) of the metric \(D\) in 
         Figure~\ref{figure:examples:or}(a). It consists
of four maximal 2-dimensional faces surrounding the vertex $k_c$, 
and three maximal 1-dimensional faces all of
which have a vertex in common (and which form the ``fork''
in the figure). 
(b) The 1-skeleton $G_D$ of $T_D$.
         (c) A weighted two-compatible split system that
             induces~\(D\).}
\label{figure:example:tight:span}
\end{figure}

We now prove two simple but useful facts
concerning the relationship between 
\(l_1\)-distances between points in
the plane, two-decomposable metrics and treelike metrics.
For a point \(p \in \mathbb{R}^2\) we denote
by \(x(p)\) and \(y(p)\) the \(x\)- and \(y\)-coordinate
of \(p\), respectively, and the \(l_1\)-distance
between two points \(p, q \in \mathbb{R}^2\) by
\(D_1(p,q) = |x(p)-x(q)| + |y(p)-y(q)|\).
Then we have:

\begin{lem}
\label{lemma:l1:is:2:decomposable}
Let \(P\) be a finite non-empty set of points
in \(\mathbb{R}^2\). Then the metric \(D_1|_P\) is
\begin{itemize}
\item[(i)] two-decomposable.
\item[(ii)] the sum of two treelike metrics.
\end{itemize}
\end{lem}

\begin{proof}
(i) Let \(\Sigma_v\) be the set of those splits \(A|B\) of \(P\)
for which there exists a real number \(r\) such that 
\(A=\{p \in P : x(p) < r\}\) and
\(B=\{p \in P : x(p) > r\}\).
Similarly, let \(\Sigma_h\) be the set of those splits \(A|B\) of \(P\)
for which there exits a real number \(r\) such that 
\(A=\{p \in P : y(p) < r\}\) and
\(B=\{p \in P : y(p) > r\}\).
For every \(S \in \Sigma_v\), put 
\(\alpha(S) = \min\{x(b) - x(a) : a \in A, b \in B\}\)
and, for every \(S \in \Sigma_h\) put
\(\beta(S) = \min\{y(b) - y(a) : a \in A, b \in B\}\).
Note that any two splits in \(\Sigma_v\) as well as
any two splits in \(\Sigma_h\) are compatible.
Hence, the split system \(\Sigma := \Sigma_v \cup \Sigma_h\)
is two-compatible.

Now, define, for any split \(S\) in \(\Sigma\), the
weight
\[
\lambda(S) =
\begin{cases}
\alpha(S), &\text{if} \ S \in \Sigma_v \setminus \Sigma_h,\\
\beta(S), &\text{if} \ S \in \Sigma_h \setminus \Sigma_v,\\
\alpha(S) + \beta(S), &\text{if} \ S \in \Sigma_h \cap \Sigma_v.
\end{cases}
\]
It is not hard to check that 
\(D_1|_P = \sum_{S \in \Sigma} \lambda(S) \cdot D_S\),
implying that \(D_1|_P\) is indeed two-decomposable.

(ii) Continuing to use the notation introduced in the proof of (i),
note that we have \(D_1|_P = D_v + D_h\) with
\(D_v = \sum_{S \in \Sigma_v} \alpha(S) \cdot D_S\) and
\(D_h = \sum_{S \in \Sigma_h} \beta(S) \cdot D_S\).
Therefore, it remains to note that
\(D_v\) and \(D_h\) are treelike in view of
the fact that a metric space \((D',X')\) is treelike 
if there exists a system \(\Sigma'\) of pairwise
compatible splits of \(X'\)
and a map \(\lambda':\Sigma' \rightarrow \mathbb{R}_{>0}\)
with \(D' = \sum_{S \in \Sigma} \lambda'(S) \cdot D_S\)
\cite{bun-71}.
\end{proof}

%%%%%%%%%%%%%%%%%%%%%%%%%%%%%%%%%%%%%%%%%%%%%%%%%%%%%%%%%%%%%%%%%%
\section{Computing a realization in the tight span}
\label{section:algorithm:tight:span}
%%%%%%%%%%%%%%%%%%%%%%%%%%%%%%%%%%%%%%%%%%%%%%%%%%%%%%%%%%%%%%%%%%

We now present our algorithm
for computing realizations using the tight span.
We also prove that it is guaranteed to yield an optimal
solution for some special types of metrics.
Given a finite metric space \((X,D)\), 
the basic idea of our algorithm is to select,
for each pair $\{x,y\}$ of distinct elements in $X$,
a shortest path from $k_x$ to $k_y$ in \(G_D\).
The union of these paths is then a realization of 
$(X,D)$. This is summarized in
the form of pseudo-code in Algorithm~\ref{alg:basic}.

\begin{algorithm}[h]
  \KwIn{A finite metric space $(X,D)$}
  \KwOut{A realization of $(X,D)$}
  Initialize the graph \(G = (V,E)\) with $V=\{k_x: x \in X\}$, $E=\emptyset$\;
  Form a list $L$ of all pairs $\{x,y\} \in \binom X2$\;
  \ForEach{$\{x,y\}\in L$}{
     \FuncSty{find\_path}($D,k_x,y,G$)\;
     \tcc{Adds, if necessary, edges of $G_D$ to \(G\) 
          so that, after the call, \(G\) contains a path of length 
          $D(x,y)$ from $k_x$ to $k_y$.}}
  \KwRet{$(G=(V,E),\omega_{D}|_E,\tau_D)$}\;
  \caption{The basic algorithm.}
  \label{alg:basic}
\end{algorithm}
\begin{algorithm}[h!]
  \SetKwInOut{Input}{Function}
  \Input{\FuncSty{find\_path}($D,u,x,G$)}
  Initialize \(v = u\)\;
  \If{$u$ {\rm is a vertex of} \(G\)}
     {Let \(M\) be the set of those vertices \(w\) of \(G\) for which\\
      there is a path of length \(D_{\infty}(u,w)\) from \(u\) to \(w\) in \(G\)\\       and \(D_{\infty}(u,x) = D_{\infty}(u,w) + D_{\infty}(w,x)\)\; 
      Let \(v\) be a vertex in \(M\) with \(D_{\infty}(v,x)\) minimum\;}
  \Else{Add \(u\) to \(G\)\;}
  \If{$v$ {\rm equals} $k_x$} {\KwRet{}\;}
  Make a simplex step from $v$ to arrive at vertex \(w\)\;
  Add the edge $\{v,w\}$ to $G$\;
  \FuncSty{find\_path}($D,w,x,G$)\;
  \caption{Compute a path using the existing partial realization.}
  \label{alg:findpath}
\end{algorithm}

Pseudocode for the function \texttt{find\_path}
is presented in Algorithm~\ref{alg:findpath}.
This function essentially computes, 
for any vertex \(u\) of \(G_D\) and any
\(x \in X\), a shortest path from \(u\) to \(k_x\) in \(G_D\).
To avoid computing the whole graph \(G_D\),
it constructs such a path
edge by edge employing the polyhedron \(P_D\) as follows.
It computes in polynomial time
from the description of \(P_D\) all vertices \(v\) of
\(G_D\) that are adjacent to \(u\) in \(G_D\). 
Among these vertices, one with
\(D_{\infty}(u,k_x) = D_{\infty}(u,v) + D_{\infty}(v,k_x)\)
that minimizes \(D_{\infty}(v,k_x)\) is selected. We refer to
this as a \emph{simplex step} from \(u\) that
\emph{arrives} at vertex \(v\), since this is
similar to one step in Dantzig's well-known 
simplex algorithm \cite{dan-63}.

To make use of the fact that certain edges
of \(G_D\) might have been added to \(G\)
in previous rounds of the \textbf{foreach}-loop
in Algorithm~\ref{alg:basic}, the function
\texttt{find\_path} first explores whether the
current graph \(G\) already contains edges
that can serve as the initial part of a suitable
path from \(u\) to \(k_x\). One would expect that
the choice of the order in which pairs are processed
in the \textbf{foreach}-loop has some impact
on how many edges can be re-used in subsequent
rounds. We found that ordering the pairs according
to increasing distances between them tends to work 
well in practice.
Then, in particular, for any elements \(x,y,z \in X\) with
\(D(x,y) + D(y,z) = D(x,z)\), no edges will be
added when processing the pair \(\{x,z\}\).

Note that our algorithm is guaranteed to output an optimal realization for
any treelike metric and any metric that corresponds to
the shortest path distances between the pairs of vertices
of a graph that is a cycle. The former follows from the
fact that, for any treelike metric, \(G_D\) \emph{is} a tree \cite{dre-84}, 
and the latter is an immediate
consequence of the fact that we process the pairs of
elements in \(X\) according to increasing distances between them.
Moreover, using the decomposition of
a given metric according to
\cite{her-var-07a,her-var-08a} as a preprocessing step,
it follows that an optimal realization can be 
obtained for a given metric \(D\)
if the decomposition of \(D\) yields only subinstances
for which our algorithm outputs an optimal realization.
In particular, it follows that our algorithm
produces optimal realizations for all inputs
given in the appendix of \cite{var-06a}.

%%%%%%%%%%%%%%%%%%%%%%%%%%%%%%%%%%%%%%%%%%%%%%%%%%%%%%%%%%%%%%%%%%%%%%
\section{Minimum Manhattan networks and optimal realizations}
\label{section:mmn:or:connection}
%%%%%%%%%%%%%%%%%%%%%%%%%%%%%%%%%%%%%%%%%%%%%%%%%%%%%%%%%%%%%%%%%%%%%%

In this section, using properties of the
tight span, we give a concise proof of
the fact that the problem of computing a minimum  
Manhattan network is nothing other than the problem of
computing an optimal realization for a special
class of finite metric spaces (see also \cite{epp-11a} 
for related work). This allows us to directly compare
our heuristic for computing realizations 
with some existing algorithms for computing minimum
Manhattan networks. Note that this fact seems
to have not been pointed out before in the literature
and has some interesting consequences for
the computational complexity of constructing an optimal
realization which we shall also point out.

To state the main result of this section,
we first introduce some more notation.
A \emph{Manhattan network} \((G=(V,E),\omega)\) 
consists of a finite graph \(G\) 
whose vertex set \(V \subseteq \mathbb{R}^2\) is a
set of points in the plane and a map \(\omega\) that
assigns to each edge \(\{p,q\} \in E\) as its length the 
\(l_1\)-distance \(D_1(p,q)\) between the points \(p\) and \(q\).
In addition, we require that, for every edge \(\{p,q\} \in E\),
the straight line segment \(\overline{p,q}\) with endpoints \(p\)
and \(q\) is either horizontal or vertical and, for any two distinct
edges \(e_1 = \{p_1,q_1\}\) and \(e_2 = \{p_2,q_2\}\) in \(E\),
the straight line segments \(\overline{p_1,q_1}\) and \(\overline{p_2,q_2}\)
do not \emph{cross}, that is, 
\(\overline{p_1,q_1} \cap \overline{p_2,q_2} \subseteq e_1 \cap e_2\). 
For any path \({\bf p}\) from \(p\) to \(q\) 
in \(G\), \(\ell({\bf p})\) denotes
the length of \({\bf p}\), and \({\bf p}\) is \emph{monotone}
if \(D_1(p,q) = \ell({\bf p})\).

Now, given a finite set of points \(P \subseteq \mathbb{R}^2\), 
a \emph{Manhattan network for} \(P\) is 
a Manhattan network \((G=(V,E),\omega)\) with \(P \subseteq V\)
such that for any two distinct \(p,q \in P\) there exists a
monotone path from \(p\) to \(q\) in~\(G\).
Such a network is called \emph{minimum} if its total length
is minimum among all Manhattan networks for~\(P\)
(cf. Figure~\ref{figure:example:mmn}).
The \emph{minimum Manhattan network} problem 
has been studied by several researchers over the
last few years (for a comprehensive list of references for this
problem see e.g. \cite{kna-spi-11a}). We have the following
relationship between minimum Manhattan networks and 
optimal realizations:

\begin{figure}
\centering
\includegraphics[scale=1.0]{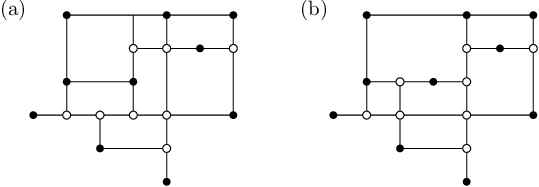}
\caption{(a) A Manhattan network for the set \(P\) of 
             points drawn as black dots. Other vertices of the
             network are drawn as empty circles.
             The network is not minimum.
         (b) A minimum Manhattan network for \(P\).}
\label{figure:example:mmn}
\end{figure}

\begin{thm}
\label{theorem:mmn:is:optimal:realization}
Let \(P\) be a finite non-empty set of points
in \(\mathbb{R}^2\). Then, for any minimum Manhattan
network \((G=(V,E),\omega)\) for \(P\),
\((G=(V,E),\omega,id_P)\) is an optimal realization
of \((P,D_1|_P)\), where \(id_P\) is the
identity map on \(P\).
\end{thm}

\begin{proof}
By definition, any Manhattan network for \(P\)
is, up to adding the map \(id_P\), 
a realization of \((P,D_1|_P)\). Hence, it
suffices to show that there exists a Manhattan
network for \(P\) whose total length
is at most the total length of some optimal
realization of \((P,D_1|_P)\).

Consider an optimal realization \((G=(V,E),\omega,\tau)\) 
of \((P,D_1|_P)\)
such that there exists an injective map \(h:V \rightarrow T_{D_1|_P}\)
with \(w(\{u,v\}) = D_{\infty}(h(u),h(v))\) for all edges
\(\{u,v\} \in E\) and \(h(\tau(p)) = k_p\) for all \(p \in P\). 
By Lemma~\ref{lemma:l1:is:2:decomposable} and 
Theorem~\ref{theorem:vertex:embedding:optimal:realization:into:tight:span},
such an optimal realization always exists.

Now, since the metric space \((\RR^2,D_1)\) is injective
(see e.g. \cite{cat-che-11a}), it follows that 
for every finite set \(P\) of points in \(\mathbb{R}^2\)
there exists an isometric embedding of \((T_{D_{1}|_P},D_{\infty})\)
into \((\mathbb{R}^2,D_1)\) that maps every \(k_p\), \(p \in P\), to~\(p\).
Therefore, there exists an injective map \(g:V \rightarrow \mathbb{R}^2\)
with \(w(\{u,v\}) = D_1(g(u),g(v))\) for all edges
\(\{u,v\} \in E\) and \(g(\tau(p)) = p\) for all \(p \in P\).
To obtain a Manhattan network for \(P\),
start with the points in \(g(V)\) and then add, step by step,  
for every \(\{u,v\} \in E\), edges to obtain
a monotone path from \(g(u)\) to \(g(v)\) (if an edge \(e\)
on this monotone path crosses an edge \(e'\) added in some
previous step, we place an additional vertex at the point
where \(e\) and \(e'\) cross to remove the crossing). Note that
in the resulting Manhattan network \(\mathcal{N}\) 
the length of a shortest path
between \(g(u)\) and \(g(v)\) can be at most the
length of a shortest path between \(u\) and \(v\) in \(G\)
for all \(u,v \in V\). This implies that there is
a monotone path from \(p\) to \(q\) in \(\mathcal{N}\)
for all \(p,q \in P\). Hence, \(\mathcal{N}\) is
indeed a Manhattan network for \(P\). Finally, the
total length of \(\mathcal{N}\) is, by construction,
not larger than the total length of \(G\), as required.
\end{proof}

Before concluding this section, we 
point out some interesting implications
of the last result:

\begin{cor}
\label{corollary:mmn:or}
Computing an optimal realization of a finite
metric space \((X,D)\) is NP-hard even if
\begin{itemize}
\item[(i)]
\(D\) is two-decomposable, or
\item[(ii)]
\(D\) is the sum of two treelike metrics on \(X\).
\end{itemize}
\end{cor}

\begin{proof}
In \cite{chi-guo-09a} it is shown that
computing (even just the total edge length of)
a minimum Manhattan network is NP-hard.
In view of Theorem~\ref{theorem:mmn:is:optimal:realization},
this implies that computing an optimal realization
of \((P,D_1|_P)\) for a given point set \(P\)
is NP-hard. By Lemma~\ref{lemma:l1:is:2:decomposable}(i) 
the metric \(D_1|_P\) is two-decomposable.
This establishes (i). Alternatively, this
also follows from the NP-hardness proof in \cite{alt-88a}: It can
be checked that the metric that arises from applying the
reduction is always two-decomposable.

In Lemma~\ref{lemma:l1:is:2:decomposable}(ii)
it was shown that \(D_1|_P\) is even the sum of 
two treelike metrics on \(P\). This establishes~(ii).
\end{proof}

%%%%%%%%%%%%%%%%%%%%%%%%%%%%%%%%%%%%%%%%%%%%%
\section{Finding minimal subrealizations}
\label{section:minimal:subrealizations}
%%%%%%%%%%%%%%%%%%%%%%%%%%%%%%%%%%%%%%%%%%%%%

In a similar spirit to finding optimal
realizations, there is a whole family
of so-called \emph{inverse shortest path} problems 
(see e.g. \cite{cui-hoc-10a} and the references therein).
In these problems, one has a collection
of allowed edit operations on graphs such as, for example,
deleting edges or, if a graph
has weights assigned to its edges, changing these weights. Each edit
operation has an associated cost. Then, given a graph \(G\)
and required distances between certain pairs
of vertices in \(G\), a minimum cost editing of \(G\)
is sought so that in the resulting graph
the shortest path distances between
the specified pairs equal the given distances. 
The problem of finding a minimal
subrealization mentioned in the introduction
can be viewed as yet another variant of this theme
and we will briefly collect some facts about it
in this section. 

First note that, in view of the fact that
the problem of computing a minimum Manhattan network
is NP-hard \cite{chi-guo-09a} 
and the fact that there is always a minimum
Manhattan network that is contained in the grid induced
by the given point set (see e.g. \cite{BenkertWWS06}),
we have:

\begin{prop}
\label{proposition:hardness:minimal:subrealizations}
The problem of computing a minimal subrealization
of a given realization \((G,\omega,\tau)\) is NP-hard
even if \(G\) is a two-dimensional grid graph.
\end{prop}

Next note that, following a similar approach
to the one used in \cite{BenkertWWS06} for computing 
a minimum Manhattan network, one can phrase
the problem of computing a minimal subrealization
as a MIP. For the convenience of the reader, we include below
the description of the MIP that we used for benchmarking
in the computational experiments and that yields, 
for any given realization \((G=(V,E),\omega,\tau)\) 
of a finite metric space \((X,D)\), a 
subgraph \(G'=(V',E')\) of \(G\) with minimum total edge length
such that \((G',\omega|_{E'},\tau)\)
is also a realization of \((X,D)\): 

\begin{itemize}
\item
For every edge \(\{u,v\} \in E\), we introduce 
two directed edges \((u,v)\) from \(u\) to \(v\) and
\((v,u)\) from \(v\) to \(u\). Let \(\overline{E}\) denote
the set of these directed edges.
\item
For every edge \(\{u,v\} \in E\), we have a binary variable
\(x_{\{u,v\}}\) indicating whether or not \(\{u,v\}\)
is an edge of \(G'\).
\item
For any two distinct elements \(x,y \in X\),
we send one unit of flow from \(\tau(x)\)
to \(\tau(y)\) that ensures that there is at least
one path from \(x\) to \(y\) 
of length \(D(x,y)\) in \(G'\). To describe this flow, 
we introduce, for every directed edge \((u,v) \in \overline{E}\),
a real-valued variable \(f_{(u,v)\{x,y\}}\). 
\item
For any two distinct elements \(x,y \in X\),
the variables must satisfy the following constraints:
\begin{itemize}
\item[(1)]
\(x_{\{u,v\}} \geq f_{(u,v)\{x,y\}} \geq 0\) and \(x_{\{u,v\}} \geq f_{(v,u)\{x,y\}} \geq 0\)
for all \(\{u,v\} \in E\).
\item[(2)]
\(\sum_{u,\{u,v\} \in E} (f_{(u,v)\{x,y\}} - f_{(v,u)\{x,y\}}) = 0\)
for all \(v \in  V \setminus \{\tau(x),\tau(y)\}\).
\item[(3)]
\(\sum_{u,\{u,v\} \in E} (f_{(u,v)\{x,y\}} - f_{(v,u)\{x,y\}}) = -1\)
for \(v = \tau(x)\).
\item[(4)]
\(\sum_{u,\{u,v\} \in E} (f_{(u,v)\{x,y\}} - f_{(v,u)\{x,y\}}) = 1\)
for \(v = \tau(y)\).
\item[(5)]
\(\sum_{\{u,v\} \in E} w(\{u,v\}) \cdot (f_{(u,v)\{x,y\}} + f_{(v,u)\{x,y\}}) \leq D_G(\tau(x),\tau(y))\).
\end{itemize}
\item
The objective function is
\[
\sum_{\{u,v\} \in E} w(\{u,v\}) \cdot x_{\{u,v\}} \rightarrow \min.
\]
\end{itemize}

In practice, we found that the size of 
the MIP can often be reduced considerably 
by only introducing the variable \(f_{(u,v)\{x,y\}}\) 
for those edges \(\{u,v\} \in E\) that actually 
lie on some shortest path from \(\tau(x)\) 
to \(\tau(y)\) in \(G\).

%%%%%%%%%%%%%%%%%%%%%%%%%%%%%%%%%%%%%%%%%%%%%
\section{Computational Experiments}
\label{section:computational:experiments}
%%%%%%%%%%%%%%%%%%%%%%%%%%%%%%%%%%%%%%%%%%%%%

To perform computational experiments, we have
implemented the algorithm described in 
Section~\ref{section:algorithm:tight:span}
in C++ as an extension to the mathematical 
software system \texttt{polymake}~\cite{polymake}.
In this implementation, we apply, as a preprocessing step, 
the decomposition of a given metric according to
\cite{her-var-07a,her-var-08a}.

The experiments are designed to 
give an impression of the range of inputs that can be attacked by our 
algorithm in terms of size and also how close the realization
produced by our algorithm is to an optimal realization. For each
size \(n\) of the ground set of the metric space,
100~randomly generated
inputs were considered and we present 
the mean run time $t$ of our algorithm (including the
preprocessing) and the mean
ratio $r_{\sg}$ between the length of the realization produced
by our algorithm and a minimal subrealization
of \((G_D,\omega_D,\tau_D)\) (if available).
The variance of these values was usually quite low and is omitted. 

\begin{table}
\begin{center}
\scriptsize
\begin{tabular}{r|rr|r|rr|r}
 $n$ & $t$ & $t_{\man}$ & $r_{\sg}$ & $t_{\TS}$ & $t_{\solve}$ & $r_{\TS}$ \\\hline
5 & 0.28 & 0.24 & 1.01 & 0.01 & 0.40 & 0.93\\%100
10 & 0.65 & 0.41 & 1.15 & 0.07 & 2.23 & 0.66\\%100
15 & 1.46 & 0.70 & 1.22 & 4.11 & 18.90& 0.55\\%100
20 & 3.07 & 1.12 & 1.27 & 254.49 & 386.28 & 0.50\\%100
25 & 7.78 & 1.47 & 1.30 & 15075.02 & 7690.96 & 0.46\\%98
30 & 11.49 & 2.04 & 1.34 &$\star$&$\star$&$\star$\\%100
35 & 21.25 & 2.84 & 1.37 &$\star$&$\star$&$\star$\\%100
40 & 37.99 & 4.03 & 1.39 &$\star$&$\star$&$\star$\\%100
45 & 64.61 & 5.68 & 1.41 &$\star$&$\star$&$\star$\\%100
50 & 105.42 & 7.89 & 1.42 &$\star$&$\star$&$\star$\\%100
55 & 167.75 & 11.27 & 1.43 &$\star$&$\star$&$\star$\\%100
60 & 256.51 & 16.66 & 1.44 &$\star$&$\star$&$\star$\\%100
65 & 379.84 & 22.79 & 1.45 &$\star$&$\star$&$\star$\\%100
70 & 555.02 & 31.90 & 1.47 &$\star$&$\star$&$\star$\\%100
75 & 791.90 & 43.60 & 1.48 &$\star$&$\star$&$\star$\\%100
80 & 1110.62 & 61.06 & 1.49 &$\star$&$\star$&$\star$\\%100
85 & 1838.51 & 116.24 & 1.50 &$\star$&$\star$&$\star$\\%100
90 & 2229.85 & 124.47 & 1.50 &$\star$&$\star$&$\star$%243
\end{tabular}
\end{center}
\caption{Results of the computational experiments for instances
         of the minimum Manhattan network problem.}
\label{table:mmn}     
\end{table}

In the tables, 
$t_{\TS}$ denotes the time to compute the whole tight span
(if the size if the tight span admitted to compute it
using \texttt{polymake}), 
$t_{\solve}$ denotes the time needed to solve the MIP 
described in Section~\ref{section:minimal:subrealizations} using 
the solver \texttt{glpksol} from the GNU linear programming kit,
and $r_{\TS}$ denotes the ratio of the length of the realization produced
by our algorithm to the total edge length of the whole 
1-skeleton of the tight span. A \(\star\)
indicates that the corresponding value could not be obtained
because the 1-skeleton of the tight span was too large or
at least too large to solve the resulting MIP.
All run times were taken on a 
Intel(R) Core(TM)2 Quad CPU \@ 2.66GHz 
machine running CentOs 5.6 using only one core.

%%%%%%%%%%%%%%%%%%%%%%%%%%%%%%%%%%%%%%%%%%%%%%%%%%%%%%%%%%%%%%%%%%%%%%%%
\subsection{Manhattan networks}
%%%%%%%%%%%%%%%%%%%%%%%%%%%%%%%%%%%%%%%%%%%%%%%%%%%%%%%%%%%%%%%%%%%%%%%%

Inputs were generated by choosing $n$ 
random points on an integer 
$10^6\times 10^6$ grid.
In addition to the MIP described in
Section~\ref{section:minimal:subrealizations}, 
we also used the MIP presented in~\cite{BenkertWWS06} 
to compute an optimal realization
for each input point set. The run time \(t_{man}\) for
solving this alternative MIP using \texttt{glpksol}
is also given in Table~\ref{table:mmn}. As can be seen,
the realizations we obtain are usually within 
a factor $c$ of the optimum that is slowly growing with $n$
reaching \(c \approx \frac{3}{2}\) for the
largest instances considered in our experiments. Note that there exist
several polynomial time algorithms that \emph{guarantee} to
produce a realization whose length is within
a constant factor of the optimum ---
currently, for the best known algorithms, 
the factor is~2 \cite{che-nou-08a,guo-sun-08a,nou-05a}.

%%%%%%%%%%%%%%%%%%%%%%%%%%%%%%%%%%%%%%%%%%%%%%%%%%%%%%%%%%%%%%%%%%%%%%
\subsection{Two-decomposable metrics}
%%%%%%%%%%%%%%%%%%%%%%%%%%%%%%%%%%%%%%%%%%%%%%%%%%%%%%%%%%%%%%%%%%%%%%

Recall that, in case the metric \(D\) is two-decomposable, 
we know that there exists an optimal realization
that is a subrealization of \((G_D,\omega_D,\tau_D)\) 
(see Section~\ref{section:two:decomposable}).
Hence, $r_{\sg}$ is actually the ratio between
the length of the realization produced by our algorithm
and the length of an optimal realization.
We tested two types of two-decomposable metrics
(cf. Table~\ref{table:two:decomposable}):

\begin{table}
\begin{center}
\scriptsize
\begin{tabular}{r|rrr|rr}
 $n$ & $t$ & $t_{\TS}$ & $t_{\solve}$ & $r_{\sg}$ & $r_{\TS}$ \\\hline
5 & 0.46 & 0.01 & 0.43 & 1.02 & 0.95\\%100
10 & 1.46 & 0.07 & 2.05 & 1.10 & 0.77\\%100
15 & 3.00 & 3.49 & 6.83 & 1.16 & 0.70\\%100
20 & 5.44 & 225.32 & 43.73 & 1.19 & 0.66\\%100
25 & 9.18 & 13174.89 & 314.37 & 1.22 & 0.63\\%100
30 & 12.87 & $\star$ & $\star$ & $\star$ & $\star$ \\%100
35 & 24.13 & $\star$ & $\star$ & $\star$ & $\star$ \\%200
40 & 38.62 & $\star$ & $\star$ & $\star$ & $\star$ \\%100
45 & 75.90 & $\star$ & $\star$ & $\star$ & $\star$ \\%100
50 & 114.40 & $\star$ & $\star$ & $\star$ & $\star$ \\%100
55 & 169.91 & $\star$ & $\star$ & $\star$ & $\star$ \\%100
60 & 250.89 & $\star$ & $\star$ & $\star$ & $\star$ \\%100
65 & 363.25 & $\star$ & $\star$ & $\star$ & $\star$ \\%100
70 & 506.94 & $\star$ & $\star$ & $\star$ & $\star$ \\%100
75 & 587.90 & $\star$ & $\star$ & $\star$ & $\star$ \\%100
80 & 844.98 & $\star$ & $\star$ & $\star$ & $\star$ \\%100
85 & 1090.04 & $\star$ & $\star$ & $\star$ & $\star$ \\%100
90 & 1319.21 & $\star$ & $\star$ & $\star$ & $\star$ \\%100
100 & 2143.58 & $\star$ & $\star$ & $\star$ & $\star$ %100
\end{tabular} \quad
\begin{tabular}{r|rrr|rr}
 $n$ & $t$ & $t_{\TS}$ & $t_{\solve}$ & $r_{\sg}$ & $r_{\TS}$ \\\hline
5 & 0.45 & 0.01 & 0.49 & 1.04 & 0.81\\%100
10 & 1.06 & 0.07 & 2.00 & 1.16 & 0.67\\%100
15 & 2.29 & 3.49 & 9.42 & 1.17 & 0.59\\%9
20 & 21.05 & 222.07 & 83.64 & 1.22 & 0.58\\%60
\end{tabular}
\end{center}
\caption{Results of the computational experiments for
         metrics that are the sum of two treelike metrics (left)
         and general two-decomposable metrics (right).}
\label{table:two:decomposable}
\end{table}

{\em Metrics that are the sum of two treelike metrics}: 
We choose two random 
binary trees with $n$ leaves, took the set of
these leaves as the ground set of the metric space and 
assigned uniformly distributed lengths (between 1 and $10^6$)
to the edges of the trees. Then we formed the sum of 
the two treelike metrics realized by the binary trees.

{\em Metrics resulting from 
random two-compatible split systems}: We generated 
random two-compatible split systems of size $2n$ by 
generating random splits and adding them to an initially empty
system if it remains two-compatible after adding the split. 
The metric considered in the experiment is the metric 
induced by the resulting split system 
where we again assigned uniformly distributed weights 
to the splits.

%%%%%%%%%%%%%%%%%%%%%%%%%%%%%%%%%%%%%%%%%%%%%%%%%%%%%%%%%%%%%%%%%%
\subsection{Random metrics}
%%%%%%%%%%%%%%%%%%%%%%%%%%%%%%%%%%%%%%%%%%%%%%%%%%%%%%%%%%%%%%%%%%

Finally, we generated random metrics on a ground set with $n$ elements 
by choosing each pairwise distance uniformly between 
$10^6$ and $2\cdot 10^6$.
The results are presented in
Table~\ref{table:general}. Note that in this experiment
it is not known whether \((G_D,\omega_D,\tau_D)\) contains
an optimal realization of the given metric as a subrealization. Therefore,
the value $r_{\sg}$ is only a lower bound on the ratio
between the length of the realization produced by our algorithm
and the length of an optimal realization.

\begin{table}
\begin{center}
\scriptsize
\begin{tabular}{r|rrr|rr}
 $n$ & $t$ & $t_{\TS}$ & $t_{\solve}$ & $r_{\sg}$ & $r_{\TS}$ \\\hline
5 & 0.48 & 0.01 & 0.53 & 1.04 & 0.90\\%100
6 & 0.72 & 0.01 & 1.80 & 1.06 & 0.74\\%100
7 & 0.88 & 0.02 & 3.46 & 1.10 & 0.56\\%100
8 & 1.35 & 0.04 & 10.12 & 1.15 & 0.39\\%100
9 & 1.31 & 0.12 & 758.07 & 1.20 & 0.26\\%100
10 & 1.32 & 0.35 & 33652.08 & 1.21 & 0.18\\%6
15 & 3.52 & 300.29 & $\star$ & $\star$ & 0.01\\%110s
25 & 15.97 & $\star$ & $\star$ & $\star$ & $\star$ \\%100
30 & 40.47 & $\star$ & $\star$ & $\star$ & $\star$ \\%100
35 & 84.81 & $\star$ & $\star$ & $\star$ & $\star$ \\%100
40 & 181.73 & $\star$ & $\star$ & $\star$ & $\star$ \\%100
45 & 330.98 & $\star$ & $\star$ & $\star$ & $\star$ \\%100
50 & 545.36 & $\star$ & $\star$ & $\star$ & $\star$ \\%100
55 & 749.25 & $\star$ & $\star$ & $\star$ & $\star$ \\%100
60 & 1204.18 & $\star$ & $\star$ & $\star$ & $\star$ \\%100
65 & 2081.53 & $\star$ & $\star$ & $\star$ & $\star$ %100
\end{tabular}
\end{center}
\caption{Results of the computational experiments for general
         metrics.}
\label{table:general}
\end{table}
%\begin{tabular}{r|rrr|rrrr|r}
% $n$ & mean time & mean time TS & mean time ILP & mean ratio & variance ratio & max ration & min ratio & mean TS ratio\\\hline
%10 & 1.39 & 25748.57 & 0.33 & 1.21 & 0.06  & 1.28  & 1.13 & 0.18\\
%5 & 0.46 & 0.45 & 0.00 & 1.04 & 0.03  & 1.15  & 1.00 & 0.90\\
%6 & 0.79 & 1.04 & 0.01 & 1.06 & 0.04  & 1.19  & 1.00 & 0.74\\
%7 & 1.14 & 2.65 & 0.02 & 1.11 & 0.05  & 1.28  & 1.01 & 0.57\\
%8 & 1.15 & 6.18 & 0.04 & 1.15 & 0.06  & 1.29  & 1.05 & 0.39\\
%9 & 1.28 & 198.04 & 0.11 & 1.23 & 0.20  & 2.93  & 1.08 & 0.26\\
%\end{tabular}

%%%%%%%%%%%%%%%%%%%%%%%%%%%%%%%%%%%%%%%%%%%%%%%%%%%%%%%%%%%%%%%%%%%%
\section{Discussion}
\label{section:discussion}
%%%%%%%%%%%%%%%%%%%%%%%%%%%%%%%%%%%%%%%%%%%%%%%%%%%%%%%%%%%%%%%%%%%%

Our computational experiments suggest that it might be
interesting to investigate whether our heuristic (or a suitable
variant of it) yields a constant-factor approximation
algorithm for computing an optimal realization, 
at least for certain classes of metrics such as,
for example, two-decomposable metrics.

We also see that our algorithm can produce realizations 
for metric spaces with up to 50 elements, even in the
case of general random metrics. Note also that all computations 
are done with arbitrary precision rationals/integers, 
to ensure combinatorial accuracy. Using floating 
point numbers instead (which would make sense at 
least for the general random metrics, that is, generic metrics) could 
further speed up the computations.

In future work, it could also be interesting to try and
develop an exact, exponential time algorithm for computing an optimal
realization. This would be helpful for
benchmarking heuristics but would also allow to check
Dress' conjecture for more examples. We expect that this
could at least give some interesting 
further insights into the structure of the problem.

%%%%%%%%%%%%%%%%%%%%%%%%%%%%%%%%%%%%%%%%%%%%%%%%%%%%%%%%%%%%%%%
\subsubsection*{Acknowledgments}
%%%%%%%%%%%%%%%%%%%%%%%%%%%%%%%%%%%%%%%%%%%%%%%%%%%%%%%%%%%%%%%

We would like to thank the two reviewers for their helpful comments.

\bibliographystyle{plain}
\bibliography{mmn-or}

%% Authors are advised to submit their bibtex database files. They are
%% requested to list a bibtex style file in the manuscript if they do
%% not want to use model2-names.bst.

%% References without bibTeX database:

% \begin{thebibliography}{00}

%% \bibitem must have one of the following forms:
%%   \bibitem[Jones et al.(1990)]{key}...
%%   \bibitem[Jones et al.(1990)Jones, Baker, and Williams]{key}...
%%   \bibitem[Jones et al., 1990]{key}...
%%   \bibitem[\protect\citeauthoryear{Jones, Baker, and Williams}{Jones
%%       et al.}{1990}]{key}...
%%   \bibitem[\protect\citeauthoryear{Jones et al.}{1990}]{key}...
%%   \bibitem[\protect\astroncite{Jones et al.}{1990}]{key}...
%%   \bibitem[\protect\citename{Jones et al., }1990]{key}...
%%   \harvarditem[Jones et al.]{Jones, Baker, and Williams}{1990}{key}...
%%

% \bibitem[ ()]{}

% \end{thebibliography}

\end{document}